\documentclass{aptpub2}
\authornames{Ronald Meester and Pieter Trapman}
\shorttitle{Spatial epidemics and a new percolation model}
\usepackage{epsfig}
\usepackage{amssymb}
\usepackage{amsmath}
\usepackage{amsfonts}

\newcommand{\ind}{1\hspace{-2.5mm}{1}}
\numberwithin{theorem}{section}
\numberwithin{definition}{section}
\numberwithin{corollary}{section}
\numberwithin{remark}{section}
\numberwithin{example}{section}
\begin{document}

\title{Bounding basic characteristics of spatial\\ epidemics with a new percolation model}

\authorone[VU University Amsterdam]{Ronald Meester} % Affiliation is just the name of your university or institution

\addressone{Vrije Universiteit Amsterdam, Department of Mathematics, De Boelelaan 1081a, 1081 HV Amsterdam, The Netherlands} % Your postal address goes here.

\authortwo[VU University Amsterdam and University Medical Center Utrecht]{Pieter Trapman} % Affiliation is just the name of your university or institution
\addresstwo{Stockholm University, Department of Mathematics, 106 91 Stockholm, Sweden}

\begin{abstract}
We introduce a new 1-dependent percolation model to describe and analyze the spread of an epidemic on a general directed and locally finite graph. We assign a two-dimensional random weight vector to each vertex of the graph in such a way that the weights of different vertices are i.i.d., but the two entries of the vector assigned to a vertex need not be independent. The
probability for an edge to be open depends on the weights of its end vertices, but conditionally on the weights, the states of the edges are independent of each other. In an epidemiological setting, the vertices of a graph represent the individuals in a (social) network and the edges represent the connections in the network. The weights assigned to an individual denote its (random) infectivity and susceptibility, respectively. We show that one can bound the percolation probability and the expected size of the cluster of vertices that can be reached by an open path starting at a given vertex from above by the corresponding quantities for independent bond percolation with a certain density; this generalizes a result of Kuulasmaa \cite{Kuul82}. Many models in the literature are special cases of our general model.
\end{abstract}

 % insert keywords separated by a semicolon

 % insert the primary Maths Subject Classification number in the first bracket
         % and the secondary ams number(s) in the second bracket
         % e.g. \ams{60E20}{49G03;49F10}

\section{Introduction, background and main results}
\label{introvacc}
We consider an extension of the standard $SIR$ (Susceptible $\to$ Infectious $\to$ Removed)
epidemic \cite{Ande99,Ande00,Diek00} on a directed graph $G=(V,E)$. Here $V$ is the (countable) vertex set
of the graph, and $E$ consists of directed edges between vertices in $V$. An edge from $u$ to $v$
is denoted by $uv$, and we say that $v$ is a (directed) neighbour of $u$.  We assume that the graph $G=(V,E)$ is simple, that is, for any $u,v \in V$ there is at most one edge from $u$ to $v$. This simplicity assumption can easily be dropped. Furthermore, we assume that the graph is locally finite, in the sense that both the in-degree and out-degree of every vertex are finite.

In a standard $SIR$ epidemic, a vertex is identified with an individual which makes (asymmetric) contacts with each of its neighbours at rate $\tau$. If an infectious individual $u$ contacts a susceptible individual $v$, then $v$ becomes infectious itself. If an individual becomes infectious it will stay infectious for a random time; the infectious periods of different individuals are i.i.d. After the infectious period an individual is removed, which can either mean that it is recovered or that the individual has died. A removed individual never becomes susceptible or infectious again. Usually, one assumes that there is one initially infectious individual $v_0$, and that all other individuals in the network are initially susceptible. Furthermore, one assumes that demography plays no role, in the sense that we ignore births, deaths not caused by the infectious disease and migration. This is a reasonable assumption if we consider emerging infectious diseases for which the time-scale of  the spread is much smaller than the time-scale of demography.

In the model just described, one implicitly assumes that all individuals in the network are the same (at least with respect to the epidemic) apart from their position in the network. In particular all individuals will have the same total infectivity and susceptibility.
In real-life however, infectivity and susceptibility  show individual variation, notably because of immunological polymorphism, or due to polymorphic reactions to vaccination \cite{Beck98,Beck02}. The infectivity and susceptibility of one individual are in general not independent. Dependencies may arise because of confounding factors, such as age, general health status or in case of sexually transmitted diseases, promiscuity and levels of condom use, which affect the susceptibility and infectivity in the same direction.

In this paper, we model heterogeneity of the population by assigning a random infectivity $W_v$ and susceptibility $\bar{W}_v$ to each vertex $v$ in $V$, where the vectors $(W_v,\bar{W}_v)$
are assumed to be i.i.d.\ and distributed as $(W,\bar{W})$, taking values in a convex subset $S$ of $\mathbb{R}_+^2 := [0,\infty)^2$.
We do not assume that $W_v$ and $\bar{W}_v$ are independent. Conditionally on the weights,
if $u$ becomes infected and $uv \in E$, then $v$ becomes infected (if it was not already) with probability
$\kappa(W_u,\bar{W}_v)$, where $\kappa$ is some connection function, specified in the model. In percolation terms, this means
that the directed edge $uv$ is open with (conditional) probability $\kappa(W_u, \bar{W}_v)$; Conditioned on the weights, the states of the edges are independent. However, without this conditioning states of edges sharing an end-vertex are dependent through the weights assigned to this common end-vertex.

We are mainly interested in
(i) the probability of a large outbreak, (ii) the probability that a disease spreads from one given individual
to another one and (iii) the expected final size of an epidemic. Percolation models have served before as useful tools to analyze these quantities; see for instance \cite{Cox88,Durr06,Kena07,Kuul82,Mill08,Newm02,Trap07,Trap10} for related material. We denote by $\mathbb{P}$ the probability
measure governing the full process of assigning weights, and making the edges open or closed.
(We do not really need to formally define the full sample space of the process.)
One necessary property
is that
\begin{equation}\label{unitinterval}
\mathbb{P}(\kappa(W_u, \bar{W}_v)\in [0,1])=1,
\end{equation}
since $\kappa(\cdot, \cdot)$ represents a probability. Sometimes it is useful
to discuss the induced measure of $\mathbb{P}$ on the space $\Omega :=\{{\rm open, closed}\}^E$, that is the induced measure on
configurations of open and closed edges.

We consider connection functions $\kappa(x,y)$, which can be written as $\kappa(xy)$, where $\kappa(z)$ is non-decreasing and concave. Examples of functions satisfying these conditions are
\begin{eqnarray}
\kappa_a(x,y) & = & xy \qquad \mbox{with $S =[0,1]^2$},\label{factorisable}\\
\kappa_b(x,y) & = & 1-e^{-\alpha xy} \qquad \mbox{with $\alpha>0$ and $S=\mathbb{R}_+^2$},\label{mooi1}\\
\kappa_c(x,y) & = & xy(\beta + xy)^{-1} \qquad \mbox{with $\beta>0$ and $S=\mathbb{R}_+^2$}.\label{ggg}
\end{eqnarray}

We note that if $\kappa(x,y)$ is {\em factorisable}, i.e.\ if there exists functions $\kappa_1(x)$ and $\kappa_2(y)$ such that $\kappa(x,y) = \kappa_1(x) \kappa_2(y)$), then we can without loss of generality assume that $\kappa(x,y)=xy$ and $S=[0,1]^2$.
This can easily be seen by first replacing $\kappa_1(W)$ by $W$ and $\kappa_2(\bar{W})$ by $\bar{W}$ and then scaling $W$ and $\bar{W}$ such that they both take values in $[0,1]$ with probability 1. This is possible because of (\ref{unitinterval}) and one can scale $W$ and $\bar{W}$ such that both are in $[0,1]$ with probability 1.
%%haakjes toegevoegd.

The connection functions $\kappa_a$, $\kappa_b$ and $\kappa_c$ have important epidemiological interpretations.
A factorisable connection function is appropriate in situations
in which there is at most one contact from an infectious individual to a given neighbour during its infectious period, or when only at the first contact of an infectious individual with a given neighbour the infection may be transmitted. This last assumption is proposed in some models for the spread of HIV \cite{Knol04,Watt92}, where the number of sexual contacts per couple can be ignored and only the number of partners is of importance. In those models the probability of an infectious contact from $u$ to $v$ is given by $\kappa(W_{u},\bar{W}_{v})=W_{u} \bar{W}_{v}$, with $W_u= 1 - \exp[-\int_0^{\Lambda_u} \tau_u(x)dx]$, where $\Lambda_u$ is the length of the infectious period of individual $u$ and $\tau_u(x)$ is the (possibly inhomogeneous) rate at which individual $u$ makes infectious contacts at time $x$ after its infection, and where $\bar{W}_v$ is the probability that individual $v$ (if still susceptible) becomes infected at an infectious contact.

The choice in (\ref{mooi1}) can be found in \cite{Brit06}, and (\ref{ggg}) is
discussed in \cite{Norr06}, both in the context of complete graphs. In neither of these two papers epidemiological interpretations of the connection functions are given.
To shed some light on a possible epidemiological interpretation of $\kappa_b$ and
$\kappa_c$, we write
$$
\kappa_b(x,y) = 1-e^{-\alpha xy} = 1- \sum_{k=0}^{\infty} \frac{(\alpha x)^k}{k!}e^{-\alpha x} (1-y)^k
$$
and
$$
\kappa_c(x,y) = \frac{xy}{\beta + xy} = 1- \sum_{k=0}^{\infty}
\frac{\beta}{x + \beta}\left(\frac{x}{x + \beta}\right)^{k}  (1-y)^k.
$$
When $0<y<1$, we see that we can interpret these connection functions as follows: for
$\kappa_c$, an infectious individual has a geometric-$\beta/(x + \beta)$ number of contacts with a neighbour, and each time, the probability that the infection is accepted is equal to $y$. For $\kappa_b$, a similar interpretation is possible,
replacing the geometric number of attempts by a Poisson-$\alpha x$ number. In
both cases, the number of attempts stochastically increases when $x$ grows, in
accordance to our interpretation of $W$ as infectivity.

It should be noted that
(\ref{mooi1}) arises when the infectious period of every individual is exponentially-$\beta$ distributed and the per neighbour infection rate of individual $u$ is $W_u$, while the probability that an infectious contact with susceptible individual $v$ leads to an infection is given by $\bar{W}_v$.
If the infectious period of individual $u$ is $\Lambda_u$, and during its infectious period it makes infectious contacts with every neighbour at rate $\alpha \tau_u$, then the number of attempts will have a Poisson-$\alpha \tau_u \Lambda_u$
distribution; hence both $\kappa_b$ and $\kappa_c$ arise naturally.
Note that replacing the combination $\kappa_c$ and $(W,\bar{W})$ by $\kappa_b$ and $(\Lambda W / \alpha,\bar{W})$, where $\Lambda$ is exponentially distributed with parameter $\beta$ and independent of $W$, does not change the induced measure on $\Omega$.

In large randomly mixing populations it is usually assumed that the contact rate of a pair of individuals scales with $n^{-1}$, where $n$ is the number of individuals in the population. With this assumption, multiple contacts of a pair of individuals are rare, and the difference between $\kappa(x,y)= 1-e^{-xy/n}$ and $\kappa(x,y)=xy/n$, is of order $n^{-2}$.

We define the usual independent bond percolation measure $\mathbb{P}^{bond}_p$ as the product measure on $\Omega$ in which edges are independently open with probability $p$ \cite{Grim99}.
If $\kappa(x,y) =xy$, $\mathbb{P}(W= W^*, \bar{W}=\bar{W}^*)=1$ and $W^*\bar{W}^* =p$, then the induced measure of $\mathbb{P}$ on  $\Omega$ is just $\mathbb{P}^{bond}_p$.
If $\kappa(x,y) =xy$ and $\mathbb{P}(W = \bar{W} =1)  = 1 - \mathbb{P}(W = \bar{W} =0)=p$, we denote the corresponding measure by $\mathbb{P}^{site}_p$. Indeed, this measure corresponds to the edge representation of independent site percolation with
parameter $p$ \cite{Grim99}, in which an edge is open if and only if both its starting and ending vertex are open. Note that
although $\mathbb{P}^{bond}_p$ is defined on $\Omega$, $\mathbb{P}^{site}_p$ is defined on the full space. This
is perhaps slightly confusing, but it works best this way.

In order to state our results, we need a few definitions.
An ordered set of edges in $E$, $\xi = (v_{0}v_{1}, v_{1}v_{2}, \ldots, v_{n-1}v_{n})$ is a (directed) path of length $n$ from $v_{0}$ to $v_{n}$. If $v_{i} \neq v_{j}$ for all $0 \leq i< j \leq n$, then the path is \textit{self-avoiding}.  It is straightforward to extend this definition to self-avoiding paths of infinite length. If we assume that a path is self-avoiding we will explicitly mention this. With some abuse of terminology we define a \textit{trivial path} as a path without edges, but with a starting vertex and the same end-vertex (i.e.\ a trivial path may be seen as a single vertex).

We say that a path is open if all edges in the path are open; a trivial path is always open. We use the notation $v_i \leadsto v_j$ if there is at least one open path from $v_i$ to $v_j$. If the final vertex of a path $\xi_1$ is the first vertex of a path $\xi_2$, we write $(\xi_1,\xi_2)$ for the {\em conjunction} of $\xi_1$ and $\xi_2$, that is, if $\xi_1 = (v_{0}v_{1}, v_{1}v_{2}, \ldots, v_{n-1}v_{n})$ and $\xi_2 = (v_{n}v_{n+1}, v_{n+1}v_{n+2}, \ldots, v_{n+m-1}v_{n+m})$, then
$$
(\xi_1,\xi_2) = (v_{0}v_{1}, v_{1}v_{2}, \ldots, v_{n-1}v_{n}, v_{n}v_{n+1}, \ldots, v_{n+m-1}v_{n+m}).
$$

For a finite or infinite path $\xi = (v_{0}v_{1}, v_{1}v_{2}, \ldots, v_{n-1}v_n,v_n v_{n+1} \cdots)$, we define the {\it truncation of $\xi$ after $n$ edges} as $\xi^{s}(n) := (v_{0}v_{1}, v_{1}v_{2}, \ldots, v_{n-1}v_n)$.
The {\it tail of $\xi$ starting after $n$ edges} is defined as $\xi^{t}(n) := (v_nv_{n+1}, \ldots)$. Both the truncation after $n$ edges and the tail starting after $n$ edges may be trivial paths.
We are now ready to specify the collections of paths we consider in this paper.

\begin{definition}
We say that a collection of paths $\Xi$ is {\em weakly hoppable}, if
for any $v \in V$, any $i,j \in \mathbb{N}$ and any two paths $\xi, \phi \in \Xi$ going through
$v$, where $v$ is the end vertex of the $i$-th edge of $\xi$ and the start vertex of the $j$-th edge of $\phi$, the conjunction $(\xi^{s}(i),\phi^{t}(j))$ is in $\Xi$ as well.
\end{definition}
In words, we need to be able to ``hop" from one path to another if they cross.
We allow for $\xi = \phi$, which implies that if a path in a weakly hoppable collection of paths $\Xi$ has a loop, then
this loop can be erased and the resulting path is still in $\Xi$.

Furthermore, let $E^{(n)}$ be the collection of the first $n$ edges in $E$, according to some given enumeration of the edges for which $\cup_{n \in \mathbb{N}}E^{(n)} = E$.  We ``approximate" $\Xi$ by sets $\Xi_n$ defined as follows: $\Xi_n$ is the collection of all infinite paths which start in $E^{(n)}$, truncated at the first instance they leave
$E^{(n)}$ together with all finite paths of which all edges are in $E^{(n)}$.
If the first edge of an infinite path $\xi$ is not in $E^{(n)}$, then the trivial path at the starting vertex of $\xi$ is in $\Xi_n$.

Finally, for a collection of paths $\Xi$ we denote by $\mathcal{C}^{\Xi}$, the event that at least one path in $\Xi$ is open. We note that for a weakly hoppable collection of paths $\Xi$, $\mathcal{C}^{\Xi}$ is equivalent to the event that at least one self-avoiding path in $\Xi$ is open, because loop-erased paths from $\Xi$ are also in $\Xi$.

\begin{definition}
We say that a collection of paths $\Xi$ is {\em hoppable} if it is weakly hoppable and if in addition,
for some enumeration of the edges,
$$
\mathcal{C}^{\Xi} = \lim_{n \to \infty} \mathcal{C}^{\Xi_n}.
$$
\end{definition}

\begin{remark}

\begin{enumerate}
\item
Most natural and useful collections of paths are hoppable. For instance, the collection of infinite paths starting at a given vertex (or, more generally, in a finite set) is hoppable, as is the collection of all paths from a given vertex $u$ to a given vertex $v$ (or, more generally, from a finite set to another finite set).
To see the first claim, we note that if $\Xi$ is the collection of all infinite paths starting at a finite set, then $\mathcal{C}^{\Xi_{n+1}} \subset \mathcal{C}^{\Xi_{n}}$ for all $n$ and $\mathcal{C}^{\Xi} = \cap_{n=1}^{\infty}\mathcal{C}^{\Xi_n}$.
If $\Xi$ is the collection of all paths from a finite set to another finite set, then $\mathcal{C}^{\Xi_{n}} \subset \mathcal{C}^{\Xi_{n+1}}$ for all $n$ and $\mathcal{C}^{\Xi} = \cup_{n=1}^{\infty}\mathcal{C}^{\Xi_n}$.
\item
If $\Xi_n$ would have been defined as ``the collection of all paths which start in $E^{(n)}$, truncated at the first instance they leave
$E^{(n)}$'', then the collection of all paths from vertex $u$ to vertex $v$ would not have been hoppable. Indeed, on $\mathbb{Z}^d$ we have that  $\mathcal{C}^{\Xi_{n+1}} \subset \mathcal{C}^{\Xi_{n}}$ for all $n$ but $\mathcal{C}^{\Xi} \neq \cap_{n=1}^{\infty}\mathcal{C}^{\Xi_n}$, since $\cap_{n=1}^{\infty}\mathcal{C}^{\Xi_n}$ occurs if there is an open paths from $u$ to $v$ or if there is an infinite open path starting at $u$.
\end{enumerate}
\end{remark}
%%not toegevoegd.

The main results of this paper are the following, generalizing results in \cite{Kuul82} (see also \cite{Mill08} for related results for the case in which $W$ and $\bar{W}$ are independent).

\begin{theorem}
\label{firstimpthm}
Let $(W,\bar{W})$ be a random vector taking values in $S$ and let $\kappa(x,y) = \kappa(xy)$ be such that for $\kappa(z)$ is increasing and concave.
Then, for every $$p \geq \kappa(\max[\mathbb{E}(W\bar{W}),\mathbb{E}(W)\mathbb{E}(\bar{W})])$$ and any hoppable collection of paths $\Xi$, we have
$$\mathbb{P}(\mathcal{C}^{\Xi}) \leq \mathbb{P}_{p}^{bond}(\mathcal{C}^{\Xi}).$$
\end{theorem}
%%definitie s' toegevoegd.

\begin{theorem}
\label{secondimpthm}
Let $\kappa(x,y)=xy$ and $S = [0,1]^2$.
Then, for every $p \leq \mathbb{E}(W\bar{W})$ and for any hoppable collection $\Xi$ of paths in $E$, we have
$$\mathbb{P}(\mathcal{C}^{\Xi}) \geq \mathbb{P}_{p}^{site}(\mathcal{C}^{\Xi}).$$
\end{theorem}

\noindent
Theorems \ref{firstimpthm} and \ref{secondimpthm} are corollaries of the forthcoming Theorem \ref{zerofuncthm}.

We will first illustrate these results by a numerical example, and then by discussing the situation on a tree. In this latter example, we
also shed some light on the reason why $E(W\bar{W})$ and $E(W)E(\bar{W})$ play an important role in our analysis.

\begin{example}
{\rm Suppose that $W = \bar{W}$ is uniformly distributed on $(a,1)$ for some $a \in (0,1)$, $G = \mathbb{L}^2$ is the square lattice with nearest neighbour (directed) edges and $\kappa(x,y)=xy$. If $a \leq \sqrt{3/4}-1/2 \approx 0.37$, then $\mathbb{E}(W\bar{W}) = \mathbb{E}(W^2) \leq 1/2$. Because the critical value for independent bond percolation is $1/2$ \cite{Grim99}, the probability that the cluster of individuals that can be reached by an open path from the origin is infinite is 0 for this model.
Let $p_c^{site}(\mathbb{L}^2) \approx 0.59$ be the critical value for site percolation on $\mathbb{L}^2$.  If $a > \frac{\sqrt{12 p_c^{site}(\mathbb{L}^2)-3}-1}{2} \approx 0.51$ then the probability that the cluster of individuals that can be reached by an open path from the origin is infinite is positive for this model.}
\end{example}

\begin{example}
{\rm
Let $G= (V,E)$ be the rooted tree in which all vertices have out-degree $d$, and in-degree 1, apart from the root $v_0$ which
has in-degree 0.
We say that the root is the generation-$0$ vertex; if there is a path of length $n$ from the root to vertex $v$, then $v$ is a said to be a generation-$n$ vertex.

Let $\kappa(x,y)=xy$, and let $\Xi_k$ be the set of all paths of length $k$ starting at the root. Furthermore, let $Z_k$ be the number of open paths in $\Xi_k$, i.e., $Z_k$ is the number of generation-$k$ vertices that can be reached by an open path from the root.
%%origin -> root.
For $k \geq 1$, and generation-1 vertex $v_1$, we define $Y_k(v_1)$ as the number of generation-$k$ vertices $v_k$, for which there is an open path from $v_1$ to $v_k$, conditioned on $W_{v_0} =1$ and $\bar{W}_{u} =1$ for all generation-$k$ vertices $u$.

We claim that $Y_k(v_1)$ is distributed as the size of the $(k-1)$-th generation of a Galton-Watson process \cite{Jage75} starting with one individual, in which individuals have no offspring with probability $\mathbb{E}(1-\bar{W}) + \mathbb{E}(\bar{W}(1-W)^d)$ and offspring of size $j$ with probability $\mathbb{E}( \bar{W} {d \choose j} W^j (1-W)^{d-j})$, for $0< j \leq d$. Indeed, we can adopt the point of view that for a vertex to have any offspring, it first has to accept the
disease from its predecessor in the tree, and if they do so, in addition need to send the disease to the next generation.
This leads to an offspring mean of $d \mathbb{E}( \bar{W}W)$.

It now seems quite natural, given the computation above, to consider a class of probability measures for which $\mathbb{E}(W\bar{W})$ is constant (compare \cite{Beck98,Beck02}). However, in order to proof the inequality  $\mathbb{P}(\mathcal{C}^{\Xi}) \leq \mathbb{P}_{p}^{bond}(\mathcal{C}^{\Xi})$ for all hoppable collections of paths in $E$, we need the additional assumptions that $p \geq \mathbb{E}(W)\mathbb{E}(\bar{W})$. This can be seen by assuming that $\Xi$ consists of a single edge. The marginal probability that this edge is open, is given by  $\mathbb{E}(W)\mathbb{E}(\bar{W})$. This explains, to some extent, the importance of the quantities $\mathbb{E}(W\bar{W})$ and
$\mathbb{E}(W)\mathbb{E}(\bar{W})$.
}
\end{example}

\section{Discussion}
\label{remarks}
Before we start proving the results, we collect in this section a number of remarks.

\begin{itemize}
\item Let $\kappa(x,y)=xy$. The (directed) edge density in our percolation model is $\mathbb{E}(W)\mathbb{E}(\bar{W})$. It is not hard to check
that under $\mathbb{P}_p^{site}$, the directed edge density is $p^2$, and under $\mathbb{P}_p^{bond}$ it is $p$. In
Theorem \ref{secondimpthm} therefore, we compare a model with edge density $\mathbb{E}(W)\mathbb{E}(\bar{W})$ with a model
with edge density at most ($\mathbb{E}(W\bar{W}))^2$. Note that the latter density is at most the former density. On the
other hand, in Theorem \ref{firstimpthm}, the edge density at the right is at least as large as the edge density at the left.

\item Contrary to Kuulasmaa \cite{Kuul82} and Miller \cite{Mill08}, we deal with infectious diseases for which the susceptibility and infectivity of individuals might be dependent. This dependence complicates the proofs and makes that we cannot apply the results from \cite{Kuul82} and \cite{Mill08} immediately. If $\kappa(x,y)$ is not factorisable, then Millers result for independent $W$ and $\bar{W}$ is slightly stronger than our result. As he provides a bond percolation upper bound with parameter $p= \mathbb{E}(\kappa(W\bar{W}))$, which for concave $\kappa$ is bounded above by our bond percolation parameter $\kappa(\mathbb{E}(W\bar{W}))$.

\item The class of possible collections $\Xi$ for which the result is true, is larger than the
class of hoppable collections. We have chosen for this formulation because it contains most collections of
interest, and also because of its elegance. One class of paths that can be seen to satisfy the results is
the class of infinite {\em backwards} paths ending at a given vertex $v$. Indeed, by simply interchanging the
role of $W$ and $\bar{W}$ (see also \cite{Mill08} for this trick), it is easy to see that our results are valid for this class as well.

\item Suppose that $\kappa(x,y)=xy$. By choosing $\Xi$ as the collection of paths from $u$ to $v$, we obtain that among all
measure with $\mathbb{E}(W\bar{W}) \geq \mathbb{E}(W) \mathbb{E}(\bar{W})$, it is the case that
$\mathbb{E}[\ind(u \leadsto v)] = \mathbb{P}(u \leadsto v)$
is at most $\mathbb{P}^{bond}_{\mathbb{E}(W\bar{W})}(u \leadsto v)$ and at least $\mathbb{P}^{site}_{\mathbb{E}(W\bar{W})}(u \leadsto v)$. Let $C_u$ be the set of vertices that can be reached by an open path from vertex $u$. The observations above give $$\mathbb{E}(|C_u|) = \sum_{v \in V} \mathbb{P}(u \leadsto v) \leq \sum_{v \in V} \mathbb{P}^{bond}_{\mathbb{E}(W\bar{W})}(u \leadsto v) = \mathbb{E}^{bond}_{\mathbb{E}(W\bar{W})}(|C_u|).$$

\item It is not possible to use straightforward stochastic domination arguments to prove the theorems in their full generality. This can be seen by considering $\kappa(x,y)=xy$ and uncorrelated $W$ and $\bar{W}$. In that case the marginal probability that any edge $uv \in E$ is open is the same for all measures for which $\mathbb{E}(W\bar{W})$ is constant. However, the edge density in $\mathbb{P}^{bond}_{\mathbb{E}(W\bar{W})}$ is also $\mathbb{E}(W\bar{W})$ and hence stochastic domination cannot be used to
prove Theorem \ref{firstimpthm}.

\item If $E$ is symmetric, i.e., $uv \in E \Leftrightarrow vu \in E$, and if $\mathbb{P}(W=\bar{W}) =1$ and $\kappa(x,y) = \kappa(y,x)$,
then the law of the cluster of vertices that can be reached by open paths from $v \in V$ on $G$ is the same as the law of the open cluster containing $v$ on the undirected counterpart of $G$ (the graph obtained by replacing the two edges connecting the same vertices by 1 undirected edge) \cite{Cox88}. Hence many questions on undirected graphs can be addressed as questions on directed graphs.

\item We can compare some of our results with bounds given in \cite{Bali05}: any undirected 1-dependent edge
percolation model on the 2-dimensional square lattice, with marginal probability for an edge to be open at least 0.8639 is supercritical. Since
this bound holds for all 1-dependent measures, it is to be expected that our bounds improve on
this in our specific model. Since we can only compare undirected models, we can only compare in the
symmetric case (which is perhaps not the most interesting one). Still, since $\mathbb{P}_p^{site}$ percolates above $p=0.68$ \cite{Wier94}
(this is a rigorous bound, the correct value of the critical probability is around $0.59$), we
see that our model percolates when $\mathbb{P}(W=\bar{W}) =1$, $\kappa(x,y)=xy$ and $\mathbb{E}(W\bar{W}) =\mathbb{E}(W^2)\geq 0.68$. Hence we improve on the general bound when $\mathbb{E}(W^2) \geq 0.68$ and $(\mathbb{E}(W))^2 < 0.8639$.

\item Our model is a generalization of many other percolation processes, such as the locally dependent random graph model \cite{Kuul82}, mixed percolation \cite{Chay00}, generalized random graphs \cite{Brit06}, and Poissonian random graphs \cite{Norr06}, where the two latter ones were previously only defined for finite complete graphs $G$.
The inhomogeneous random graphs of \cite{Boll07} are only defined on complete graphs, but they are more general than our model on the complete graph.
\item Theorems \ref{firstimpthm} and \ref{secondimpthm} can easily be generalised to models where the vectors $(W_v,\bar{W}_v)$ are independent, but not necessarily identically distributed. Furthermore, the assumption that $\kappa(x,y) = \kappa(xy)$ in  Theorem \ref{firstimpthm} may be replaced by: for every $x$, $\kappa(x,y)$ is non-decreasing and concave in $y$ and for every $y$, $\kappa(x,y)$ is non-decreasing and concave in $x$.
\end{itemize}

\section{Proofs}
\label{factres}
Let $S_1$ be the projection of $S$ in the first coordinate direction, and
$S_2$ the projection of $S$ in the second coordinate direction.
Let $E'_v$ be the set of all edges starting at $v$, and $E^*_v$ the set of all edges ending at $v$.
For any pair of (possibly empty) sets $A \subset E'_u$ and $B \subset E^*_u$, any $|A|$-dimensional vector $\mathbf{x} = (x_1,x_2,\ldots,x_{|A|}) \in (S_2)^{|A|}$ and any $|B|$-dimensional vector $\mathbf{y} = (y_1,y_2, \ldots , y_{|B|}) \in (S_1)^{|B|}$ we define the \textit{zero function} $z_u(\mathbb{P};A,B; \mathbf{x},\mathbf{y})$ as the probability that either none of the edges in $A$ are open or none of the edges in $B$ are open if the weights assigned to the endpoints of the edges in $A$ (resp.\ $B$) are the elements of the vector $\mathbf{x}$ (resp.\ $\mathbf{y}$).

In formulas this reads
$$
z_{u}(\mathbb{P}; A,B; \mathbf{x},\mathbf{y}) :=  \mathbb{E}\left(1 - [1-\prod_{i=1}^{|A|}(1-\kappa(W,x_i))][1-\prod_{j=1}^{|B|}(1-\kappa(y_j,\bar{W}))]\right),
$$
if $|A||B|>0$. For the remaining cases we define
\begin{eqnarray*}
z_{u}(\mathbb{P}; A,\emptyset ; \mathbf{x},\emptyset) & := & \mathbb{E}\left(\prod_{i=1}^{|A|}(1-\kappa(W,x_i))\right), \, \mbox{if $|A|>0$}, \\
z_{u}(\mathbb{P}; \emptyset,B; \emptyset,\mathbf{y}) & := & \mathbb{E}\left(\prod_{j=1}^{|B|}(1-\kappa(y_j,\bar{W}))\right), \, \mbox{if $|B|>0$},\\
z_{u}(\mathbb{P}; \emptyset ,\emptyset; \emptyset ,\emptyset) & := & 1.
\end{eqnarray*}
If the graph is transitive we do not need the reference to the vertex $u$ in the zero function.

In the epidemiological setting, the zero function $z_u(\mathbb{P};A,B; \mathbf{x},\mathbf{y})$ is the (conditional) probability that if all endpoints of edges in $B$ become infected and have ``$W$-weights'' $y_1,\ldots , y_{|B|}$, either $u$ will not get infected via an edge in $B$, or $u$ will not transmit the disease to any of the endpoints of edges in $A$, if those endpoints have ``$\bar{W}$-weights'' $x_1,\ldots ,x_{|A|}$.

We write $z_{v}(\mathbb{P}^{(a)}) \geq z_{v}(\mathbb{P}^{(b)})$ if $z_{v}(\mathbb{P}^{(a)};A,B;\mathbf{x},\mathbf{y}) \geq z_{v}(\mathbb{P}^{(b)};A,B;\mathbf{x},\mathbf{y}))$ for all $A \subset E_u$, all $B\subset E^*_u$, all $\mathbf{x} \in (S_2)^{|A|}$ and all $\mathbf{y} \in (S_1)^{|B|}$. The following result is interesting in its own right, and will be the main tool
to prove Theorem \ref{firstimpthm}.

\begin{theorem}\label{zerofuncthm}
If $z_{v}(\mathbb{P}^{(a)}) \leq z_{v}(\mathbb{P}^{(b)})$ for all $v \in V$, then for any
hoppable collection $\Xi$ of paths,
\begin{displaymath}
\mathbb{P}^{(b)}(\mathcal{C}^{\Xi}) \leq \mathbb{P}^{(a)}(\mathcal{C}^{\Xi}).
\end{displaymath}
\end{theorem}

\begin{remark}
\label{tegenvoorbeeld}
{\rm Theorem \ref{zerofuncthm} does not hold if we would allow for all collections of paths $\Xi$. Indeed, here is a counterexample.
Let $G$ be the subgraph of the 2-dimensional square lattice, consisting of the origin and its nearest neighbours (with nearest neighbour edges).
Let $\kappa(x,y) = xy$ and the weights assigned to the neighbours of the origin are all equal to 1.
We consider two measures $\mathbb{P}^{(a)}$ and $\mathbb{P}^{(b)}$ on the weights assigned to the origin:

\noindent
$\mathbb{P}^{(a)}(W=\bar{W}=0) =3/5$, $\mathbb{P}^{(a)}(W= 1/2, \bar{W}=1) = \mathbb{P}^{(a)}(W= 1, \bar{W}=1/2) =1/5$, and\\
$\mathbb{P}^{(b)}(W= 0, \bar{W}=1/2) = \mathbb{P}^{(b)}(W= 0, \bar{W}=1) = \mathbb{P}^{(b)}(W= 1/2, \bar{W}=0) = \mathbb{P}^{(b)}(W= 1, \bar{W}=0) = \mathbb{P}^{(b)}(W= 1, \bar{W}=1) = 1/5$.

A small computation shows that for $|A||B| \geq 1$,
\begin{multline*}
z_0(\mathbb{P}^{(a)};A,B;\mathbf{x},\mathbf{y}) =  1 - (1/5) (2^{-|A|} + 2^{-|B|}) \prod_{i=1}^{|A|} x_i  \prod_{j=1}^{|B|} x_j \geq\\
\geq  =   1 - (1/5) \prod_{i=1}^{|A|} x_i  \prod_{j=1}^{|B|} x_j   = z_0(\mathbb{P}^{(b)};A,B;\mathbf{x},\mathbf{y}),
\end{multline*}
and for $|A| >0$,
\begin{multline*}
z_0(\mathbb{P}^{(a)}; A,\emptyset; \mathbf{x},\emptyset) = 3/5 + (1/5)  \prod_{i=1}^{|A|} (1-x_i/2) + (1/5)  \prod_{i=1}^{|A|} (1-x_i) \geq\\
\geq 2/5 + (1/5)  \prod_{i=1}^{|A|} (1-x_i/2) + (2/5)  \prod_{i=1}^{|A|} (1-x_i) = z_0(\mathbb{P}^{(b)}; A,\emptyset; \mathbf{x},\emptyset),
\end{multline*}
while for $|B| > 0$,
\begin{multline*}
z_0(\mathbb{P}^{(a)}; \emptyset, B; \emptyset,\mathbf{y}) = 3/5 + (1/5)  \prod_{i=1}^{|B|} (1-y_i/2) + (1/5)  \prod_{i=1}^{|B|} (1-y_i) \geq\\
\geq 2/5 + (1/5)  \prod_{i=1}^{|B|} (1-y_i/2) + (2/5)  \prod_{i=1}^{|B|} (1-y_i) = z_0(\mathbb{P}^{(b)}; \emptyset, B; \emptyset,\mathbf{y}).
\end{multline*}
Hence we have that $z_0(\mathbb{P}^{(a)}) \geq z_0(\mathbb{P}^{(b)})$.

Now let $\xi$ be the path from $(0,-1)$ to $(0,1)$, let $\phi$ be the path from $(-1,0)$ to $(1,0)$ in $G$ and let
$\Xi = \{\xi,\phi\}$; this is not a hoppable collection. Note that
$\mathbb{P}(\mathcal{C}^{\Xi}) = 1- \mathbb{E}([1-W\bar{W}]^2)$, and a quick computation yields
$$\mathbb{P}^{(a)}(\mathcal{C}^{\Xi}) = 3/10 > 1/5 = \mathbb{P}^{(b)}(\mathcal{C}^{\Xi}).$$}
\end{remark}

In the proof of  Theorem \ref{zerofuncthm} we need the following Lemma:

\begin{lemma}
If $\Xi$ is hoppable, then $\Xi_n$ is hoppable for all $n \in \mathbb{N}$.
\end{lemma}

\begin{proof}
Since $\Xi_n$ consists of finitely many edges only,
it is enough to prove that $\Xi_n$ is weakly hoppable, that is,
for any $v$ and any two paths $\xi_n, \phi_n \in \Xi_n$ going through
$v$, (say that $v$ is the end vertex of the $i$-th edge of $\xi$ and the start vertex of the $j$-th edge of $\phi$),  the conjunction $(\xi_n^{s}(i),\phi_n^{t}(j))$ is in $\Xi_n$ as well. We distinguish various cases.
\begin{enumerate}
\item If $\phi_n$ is the truncations of the infinite path $\phi \in \Xi$, then $(\xi_n^{s}(i),\phi_n^{t}(j))$ is the truncation of $(\xi_n^{s}(i),\phi^{t}(j))$  at the first time this paths leaves $E_n$. All of its edges are in $E_n$. By the weak hoppability of $\Xi$ this conjunction is in $\Xi_n$ as well.
\item If $\xi_n$ is the truncation of the infinite path $\xi \in \Xi$ and $\phi_n \in \Xi$, then $(\xi_n^{s}(i),\phi_n^{t}(j))$ contains only edges in $E_n$ and is in $\Xi$ by the weak hoppability of $\Xi$ and in $\Xi_n$ by the definition of $\Xi_n$.
\item If $\xi_n$ and $\phi_n$ are both finite paths in $\Xi$, then $(\xi_n^{s}(i),\phi_n^{t}(j))$ contains only edges in $E_n$ and is in $\Xi$ by the weak hoppability of $\Xi$, and in $\Xi_n$ by definition.
\end{enumerate}
\end{proof}

\noindent \begin{proof}[Proof of Theorem \ref{zerofuncthm}]

$\textbf{(i)}$ The first step is to prove that for all $n \geq 1$, $\mathbb{P}^{(a)}(\mathcal{C}^{\Xi_n}) \geq \mathbb{P}^{(b)}(\mathcal{C}^{\Xi_n})$, if there is a $u \in V$ such that $z_{u}(\mathbb{P}^{(a)}) \leq z_{u}(\mathbb{P}^{(b)})$, and such that $z_{v}(\mathbb{P}^{(a)}) = z_{v}(\mathbb{P}^{(b)})$ for all $v \in V \setminus \{u\}$.

Let $E^{(n)}_u$ be the set of edges in $E^{(n)}$ with $u$ as start or end vertex.
Let $\sigma_u^{(n)}$ be a realisation of the weights outside $u$, together with the states (open or closed) of the edges in
$E^{(n)}\backslash E^{(n)}_u$. The space of configuration restricted to these quantities is denoted  $\hat{\Omega}_u^{(n)}$.
Informally, $\sigma_u^{(n)}$ contains all information which does not depend on $(W_u,\bar{W}_u)$.

Now there are trhee possibilities.
\begin{enumerate}
\item  $\sigma_u^{(n)}$ is such that, no matter what the states (open or closed) of the edges in $E^{(n)}_u$ is, no path in $\Xi_n$ will be open.
\item $\sigma_u^{(n)}$ is such that it contains an open path in $\Xi_n$ of which none of the edges has $u$ as a starting or end-vertex.
\item Conditioned on $\sigma_u^{(n)}$, if all edges in $E^{(n)}_u$ are open, then there is an open path in $\Xi_n$, while if they are all closed, there is no such open path.
\end{enumerate}
The final case may be split into three further cases, where we use that we only have to consider self-avoiding paths, because if $\Xi_n$ is weakly hoppable, loop erased paths of non self-avoiding paths in $\Xi_n$ are also in $\Xi_n$.

\begin{itemize}
\item[(i)]If $u$ is the starting vertex of a self-avoiding path in $\Xi_n$, then there is a set $A \subset E^{(n)}_u$, such that if at least one of the edges in $A$ is open, then there will be an open path in $\Xi_n$, while if the only open edges in  $E^{(n)}_u$ are  $E^{(n)}_u \setminus A$, then there will be no open path in $\Xi_n$.
\item[(ii)] if $u$ is the end-vertex of a path in $\Xi_n$, then there is a set $B \subset E^{(n)}_u$, such that if at least one of the edges in $B$ is open, then there will be an open path in $\Xi_n$, while if the only open edges in  $E^{(n)}_u$ are  $E^{(n)}_u \setminus B$, then there will be no open path in $\Xi_n$.
\item[(iii)] If $u$ is neither the starting nor the end vertex of a path in $\Xi_n$, then there are sets $A,B \subset  E^{(n)}_u$ such that if at least one edge in $A$ is open and at least one edge in $B$ is open, then their will be an open path in $\Xi_n$, while if either all edges in $A$ are closed or all edges in $B$ are closed, then  there will be no open path in $\Xi_n$.
Here $A$ is the set of all starting edges of tails of paths in $\Xi_n$ cut at vertex $u$, of which all edges, apart from the starting edge are open. While $B$ is the set of all final edges of truncations of paths in $\Xi_n$ cut at vertex $u$, of which all edges, apart from the final edge are open.
\end{itemize}

Observe that for $i=a,b$,
$$
\mathbb{P}^{(i)}(\mathcal{C}^{\Xi_n}) = \int_{\hat{\Omega}_u^{(n)}} \mathbb{P}^{(i)}(\mathcal{C}^{\Xi_n}|\sigma_u^{(n)}) d\mathbb{P}^{(i)}(\sigma_u^{(n)}).
$$
By the assumption $z_{v}(\mathbb{P}^{(a)}) = z_{v}(\mathbb{P}^{(b)})$ for all $v \in V \setminus \{u\}$, this implies that
\begin{eqnarray*}
\mathbb{P}^{(a)}(\mathcal{C}^{\Xi_n}) - \mathbb{P}^{(b)}(\mathcal{C}^{\Xi_n})
& = & \displaystyle\int_{\hat{\Omega}_u^{(n)}} (\mathbb{P}^{(a)}(\mathcal{C}^{\Xi_n}|\sigma_u^{(n)})-
\mathbb{P}^{(b)}(\mathcal{C}^{\Xi_n}|\sigma_u^{(n)}))d\mathbb{P}^{(a)}(\sigma_u^{(n)}).
\end{eqnarray*}

We now split the space of all possible realisations of $\sigma_u^{(n)}$ into the cases discussed above.

\begin{enumerate}
\item $\sigma_u^{(n)}$ is such that $\mathbb{P}^{(a)}(\mathcal{C}^{\Xi_n}|\sigma_u^{(n)}) = \mathbb{P}^{(b)}(\mathcal{C}^{\Xi_n}|\sigma_u^{(n)}) =0$, in which case
$$
(\mathbb{P}^{(a)}(\mathcal{C}^{\Xi_n}|\sigma_u^{(n)})-
\mathbb{P}^{(b)}(\mathcal{C}^{\Xi_n}|\sigma_u^{(n)}))= 0.
$$
\item $\sigma_u^{(n)}$ is such that $\mathbb{P}^{(a)}(\mathcal{C}^{\Xi_n}|\sigma_u^{(n)}) = \mathbb{P}^{(b)}(\mathcal{C}^{\Xi_n}|\sigma_u^{(n)}) =1$, in which case
$$
(\mathbb{P}^{(a)}(\mathcal{C}^{\Xi_n}|\sigma_u^{(n)})-
\mathbb{P}^{(b)}(\mathcal{C}^{\Xi_n}|\sigma_u^{(n)}))= 0.
$$
\item
$\sigma_u^{(n)}$ is such that neither $\mathbb{P}^{(a)}(\mathcal{C}^{\Xi_n}|\sigma_u^{(n)}) = \mathbb{P}^{(b)}(\mathcal{C}^{\Xi_n}|\sigma_u^{(n)}) =1$ nor
$\mathbb{P}^{(a)}(\mathcal{C}^{\Xi_n}|\sigma_u^{(n)}) = \mathbb{P}^{(b)}(\mathcal{C}^{\Xi_n}|\sigma_u^{(n)}) =0$.
\end{enumerate}
In the last case there are three possibilities:
\begin{itemize}
\item[(i)] If $u$ is the starting vertex of paths in $\Xi_n$, then - since $\Xi_n$ is weakly hoppable - the tail starting at $u$ of a self-avoiding path through $u$, is also in $\Xi_n$.
Therefore, with the assumptions on $\sigma_u^{(n)}$, $\Xi_n$ contains an open path, if and only if it contains an open path starting at $u$.
Let $A \subset E_u^{(n)}$ be the set of starting edges of open paths in $\Xi_n$ starting at $u$, if all edges in $E_u^{(n)}$ would have been open.
Since, $z_{v}(\mathbb{P}^{(a)};A,\emptyset;\mathbf{x},\emptyset) \leq z_{v}(\mathbb{P}^{(b)};A,\emptyset;\mathbf{x},\emptyset))$
for all $A$ and $\mathbf{x}$, it follows that
$$
\mathbb{P}^{(a)}(\mathcal{C}^{\Xi_n}|\sigma_u^{(n)})-
\mathbb{P}^{(b)}(\mathcal{C}^{\Xi_n}|\sigma_u^{(n)}) \geq 0.
$$
\item[(ii)] If $u$ is the end vertex of paths in $\Xi_n$, then - since $\Xi_n$ is weakly hoppable - the truncation at vertex $u$ of a self-avoiding path through $u$, is also in $\Xi_n$.
Therefore, with the assumptions on $\sigma_u^{(n)}$, $\Xi_n$ contains an open path, if and only if it contains an open path ending at $u$.
Let $B \subset E_u^{(n)}$ be the set of end edges of open paths in $\Xi_n$ ending at $u$, if all edges in $E_u^{(n)}$ would have been open.
Since $z_{v}(\mathbb{P}^{(a)};\emptyset,B;\emptyset,\mathbf{y}) \leq
z_{v}(\mathbb{P}^{(b)};\emptyset,B;\emptyset,\mathbf{y}))$
for all $B$ and $\mathbf{y}$, it follows that
$$
\mathbb{P}^{(a)}(\mathcal{C}^{\Xi_n}|\sigma_u^{(n)})-
\mathbb{P}^{(b)}(\mathcal{C}^{\Xi_n}|\sigma_u^{(n)}) \geq 0.
$$
\item[(iii)] If $u$ is neither a starting vertex nor an end vertex of paths in $\Xi_n$, then
let $A$ be the set of starting edges of tails of self-avoiding paths in $\Xi_n$, cut off at $u$, for which all edges in the tail are open, if all edges in $E_u^{(n)}$ would have been open. Similarly, let
$B$ be the set of end edges of truncations of self-avoiding paths in $\Xi_n$, cut off at $u$, for which all edges in the truncation are open, if all edges in $E_u^{(n)}$ would have been open.
Since $\Xi_n$ is weakly hoppable, it follows that there is no open path in $\Xi_n$ if either all edges in $A$ are closed or all edges in $B$ are closed.
By $z_{v}(\mathbb{P}^{(a)};A,B;\mathbf{x},\mathbf{y}) \leq z_{v}(\mathbb{P}^{(b)};A,B;\mathbf{x},\mathbf{y}))$
for all $A$, $B$, $\mathbf{x}$ and $\mathbf{y}$, it follows that
$$
(\mathbb{P}^{(a)}(\mathcal{C}^{\Xi_n}|\sigma_u^{(n)})-
\mathbb{P}^{(b)}(\mathcal{C}^{\Xi_n}|\sigma_u^{(n)})) \geq 0.
$$
\end{itemize}
This concludes the first step of the proof.

$\textbf{(ii)}$ In this second step, we relax the condition that the zero functions differ in one
place only. Since the event $\mathcal{C}^{\Xi_n}$ depends on the weights of at most $2n$ vertices in $V$,
it is straightforward to construct a sequence of probability measures,
$(\mathbb{P}^{(i)};1 \leq i \leq 2n)$,  such that $\mathbb{P}^{(1)}(\mathcal{C}^{\Xi_n})= \mathbb{P}^{(a)}(\mathcal{C}^{\Xi_n})$, $\mathbb{P}^{(2n)}(\mathcal{C}^{\Xi_n})= \mathbb{P}^{(b)}(\mathcal{C}^{\Xi_n})$ and such that two subsequent zero functions $z_v(\mathbb{P}^{(i)})$  and $z_v(\mathbb{P}^{(i+1)})$ for $1 \leq i \leq 2n-1$ differ at only one vertex $v_i \in V$. Repeatedly applying part $(i)$ finishes this part of the proof.

$\textbf{(iii)}$ To finish the proof, we simply note that from the definition of hoppable, we have
for $i=a,b$,
$$
\mathbb{P}^{(i)}(\mathcal{C}^{\Xi})=\lim_{n \to \infty}\mathbb{P}^{(i)}(\mathcal{C}^{\Xi_n}).
$$
and the result follows.
\end{proof}

For the proof of Theorem \ref{firstimpthm} we need the following elementary fact.
\begin{lemma}
\label{convexlemma}
Let $f_i(x)$, $i=1,\ldots, n$ be convex non-increasing, non-negative functions on some domain $D \subset \mathbb{R}_+$. Then $\prod_{i=1}^n f_i(x)$ is also a convex non-increasing, non-negative function on $D$.
\end{lemma}
%%D iets gespecificeerd.
\begin{proof}
For $n=2$ observe that for $0<c<1$ and $y>x$
\begin{eqnarray*}
c f_1(x)f_2(x) + (1-c)f_1(y)f_2(y)& = & [c f_1(x) + (1-c)f_1(y)][c f_2(x) + \\
& & + (1-c)f_2(y)] + c(1-c)[f_1(y)-f_1(x)][f_2(y)-f_2(x)]\\
& \geq & f_1(cx +(1-c)y)f_2(cx+(1-c)y),
\end{eqnarray*}
where in the inequality, we have used that both $f_1$ and $f_2$ are non-increasing and convex. The proof of the lemma can be completed
with induction on $n$ - we leave the details to the reader.
\end{proof}

\noindent
\begin{proof}[Proof of Theorem \ref{firstimpthm}]
Let $(W^*,\bar{W}^*) \in S$ be such that $\mathbb{E}(W) \leq  W^*$, $\mathbb{E}(\bar{W}) \leq  \bar{W}^*$ and $\mathbb{E}(W \bar{W}) \leq W^*\bar{W}^*$. Let $\hat{\mathbb{P}}^{bond}$ be a measure that satisfies $\hat{\mathbb{P}}^{bond}(W = W^*,\bar{W} = \bar{W}^*) =1$.

We proceed by proving that for all $A$, $B$, $\mathbf{x}$ and $\mathbf{y}$,
$$z_v(\mathbb{P};A,B; \mathbf{x},\mathbf{y}) \geq z_v(\hat{\mathbb{P}}^{bond};A,B; \mathbf{x},\mathbf{y}).$$
From Theorem \ref{zerofuncthm} and the observation that the induced measure of $\hat{\mathbb{P}}^{bond}$ on $\Omega$ is just $\mathbb{P}_p^{bond}$, where $p=\kappa(W^*\bar{W}^*)$, Theorem \ref{firstimpthm} then follows.

If $|B|>0$, then $z_v(\mathbb{P}; \emptyset,B; \emptyset,\mathbf{y}) = \mathbb{E}(\prod_{i=1}^{|B|} [1- \kappa(y_i \bar{W})])$.
The functions $1- \kappa(y_i \bar{W})$ are by assumption non-negative, convex and decreasing. Therefore, by Lemma \ref{convexlemma} and Jensen's inequality
$$ \mathbb{E}(\prod_{i=1}^{|B|} [1- \kappa(y_i \bar{W})]) \geq \prod_{i=1}^{|B|} [1- \kappa(y_i \mathbb{E}(\bar{W}))] \geq  \prod_{i=1}^{|B|} [1- \kappa(y_i \bar{W}^*)],$$
which is equal to $z_v(\hat{\mathbb{P}}^{bond};\emptyset,B; \emptyset,\mathbf{y})$.
For $|A|>0$, we can use similar arguments to show that $$z_v(\mathbb{P};A,\emptyset; \mathbf{x},\emptyset) \geq z_v(\hat{\mathbb{P}}^{bond};A,\emptyset; \mathbf{x},\emptyset).$$

If $|A||B| \geq 1$, then we need to show that
$$
\mathbb{E}\Bigl((1-\prod_{i=1}^{|A|}[1-\kappa(Wx_i)])(1-\prod_{j=1}^{|B|}[1-\kappa(\bar{W}y_j)]) \Bigr) \leq
(1-\prod_{i=1}^{|A|}[1-\kappa(W^*x_i)])(1-\prod_{j=1}^{|B|}[1-\kappa(\bar{W}^*y_j)]).
$$
By the assumption that $\kappa(z)$ is increasing concave and taking values in $[0,1]$ and by Lemma \ref{convexlemma} it follows
that $1-\prod_{i=1}^{|A|}[1-\kappa(x_iu)]$ is concave and increasing in $u$.
By this concavity and $\kappa(0) \geq 0$ it is possible to chose $a := a(\mathbf{x}) \geq 0$ and $b := b(\mathbf{x}) \geq 0$ such that $a +b u \geq 1-\prod_{i=1}^{|A|}[1-\kappa(x_iu)]$ for all $u \in S_1$ and
$a +b W^* = 1-\prod_{i=1}^{|A|}[1-\kappa(x_iW^{*})]$.
Similarly,  $1-\prod_{j=1}^{|B|}[1-\kappa(y_j u)]$ is concave and increasing in $u$, and it is possible to chose
$\bar{a} := \bar{a}(\mathbf{y}) \geq 0$ and $\bar{b} := \bar{b}(\mathbf{y}) \geq 0$ such that $\bar{a} +\bar{b} u \geq 1-\prod_{j=1}^{|B|}[1-\kappa(y_j u)]$ for all $u \in S_1$ and
$\bar{a} +\bar{b} \bar{W}^* = 1-\prod_{ij=1}^{|B|}[1-\kappa(y_j\bar{W}^{*})]$.

It follows that
\begin{eqnarray*}
\ & \ & \mathbb{E}\Bigl((1-\prod_{i=1}^{|A|}[1-\kappa(Wx_i)])(1-\prod_{j=1}^{|B|}[1-\kappa(\bar{W}y_j)]) \Bigr) \\
\ & \leq & \mathbb{E}\Bigl((a+ bW)(\bar{a} +\bar{b}\bar{W} \Bigr) \\
\ & = & a \bar{a} + a \bar{b} \mathbb{E}(\bar{W}) + \bar{a} b \mathbb{E}(W) + b \bar{b} \mathbb{E}(W\bar{W})\\
\ & \leq & ab + a \bar{b} \bar{W}^* + \bar{a} b W^* +b \bar{b} W^* \bar{W}^*\\
\ & = & (a + bW^*)(\bar{a} + \bar{b} \bar{W}^*)\\
\ & = & (1-\prod_{i=1}^{\infty}[1-\kappa(W^*x_i)])(1-\prod_{j=1}^{|B|}[1-\kappa(\bar{W}^*y_j)]),
\end{eqnarray*}
where the second inequality follows from the definitions of $W^*$ and $\bar{W}^*$.
This finishes the proof of Theorem \ref{firstimpthm}.
\end{proof}

\begin{proof}[Proof of Theorem \ref{secondimpthm}]

In order to use Theorem \ref{zerofuncthm}, we have to show that for all $\mathbb{P}$, $A,B \subset E^{(n)}_v$, all $\mathbf{x} \in (S_1)^{|A|}$ and all $\mathbf{y} \in (S_2)^{|B|}$ we have
$$z_v(\mathbb{P};A,B; \mathbf{x},\mathbf{y}) \leq z_v(\mathbb{P}_p^{site};A,B;\mathbf{x},\mathbf{y}),$$
where $p = \mathbb{E}(W\bar{W})$.

For $\kappa(x,y) =xy$, this means that we have to prove that for $|A|\geq 1$ and $|B| \geq 1$ it holds that
\begin{eqnarray*}
\mathbb{E}\Bigl(1- [1-\prod_{i=1}^{|A|} (1-Wx_i)] [1-\prod_{j=1}^{|B|} (1-y_j\bar{W})] \Bigr) & \geq & 1- \mathbb{E}(W\bar{W}) [1-\prod_{i=1}^{|A|} (1-x_i)] [1-\prod_{j=1}^{|B|} (1-y_j)],\\
\mathbb{E}\Bigl(\prod_{i=1}^{|A|} (1-Wx_i)]\Bigr) & \geq & \mathbb{E}(W\bar{W}) [1-\prod_{i=1}^{|A|} (1-x_i)],\\
\mathbb{E}\Bigl(\prod_{j=1}^{|B|} (1-y_j\bar{W}) \Bigr) & \geq & \mathbb{E}(W\bar{W}) [1-\prod_{j=1}^{|B|} (1-y_j)],
\end{eqnarray*}
where respectively the cases $|A||B|>0$, $B= \emptyset$ and $A= \emptyset$ are considered.

It is easy to check by induction that for $u \in [0,1]$ and $\mathbf{x} \in [0,1]^{|A|}$, it holds that  $1-\prod_{i=1}^{|A|} (1-ux_i) \geq u[1-\prod_{i=1}^{|A|} (1-x_i)]$.
Similarly  for $\bar{u} \in [0,1]$ and $\mathbf{y} \in [0,1]^{|B|}$, it holds that  $1-\prod_{j=1}^{|B|} (1-\bar{u}y_j) \geq \bar{u}[1-\prod_{j=1}^{|B|} (1-y_j)]$. Substituting this with $u$ replaced by $W$ and $\bar{u}$ by $\bar{W}$, finishes the proof of Theorem \ref{secondimpthm}.
\end{proof}

\section*{Acknowledgments}
We want to thank an anonymous referee for useful comments which lead to a more elegant proof of Theorem 1.1.
The presented research is supported in part by a Vici grant of the NWO (Dutch Organization for Scientific Research).

\end{document}